\documentclass[10pt]{article}
\font\smallit=cmti10
\font\smalltt=cmtt10

\usepackage{amssymb,latexsym,amsmath,epsfig,amsthm,mathrsfs,mathtools,mathdots,tikz-cd,quiver,hyperref} 
\usepackage[framemethod=TikZ]{mdframed}

\makeatletter

\renewcommand\section{\@startsection {section}{1}{\z@}
{-30pt \@plus -1ex \@minus -.2ex}
{2.3ex \@plus.2ex}
{\normalfont\normalsize\bfseries\boldmath}}

\renewcommand\subsection{\@startsection{subsection}{2}{\z@}
{-3.25ex\@plus -1ex \@minus -.2ex}
{1.5ex \@plus .2ex}
{\normalfont\normalsize\bfseries\boldmath}}

\renewcommand{\@seccntformat}[1]{\csname the#1\endcsname. }

\makeatother

\usepackage{titlesec}

\newtheorem{theorem}{Theorem}[section]
\newtheorem{proposition}[theorem]{Proposition}

\titleformat{\section}{\normalfont\Large\bfseries}{\S\thesection}{1em}{}[]
\numberwithin{equation}{section}

\newtheorem{lemma}[theorem]{Lemma}
\newtheorem{definition}[theorem]{Definition}

\newtheorem{remark}[theorem]{Remark}
\newtheorem{example}[theorem]{Example}
\newtheorem{fact}[theorem]{Fact}
\newtheorem{question}[theorem]{Question}

\newcommand{\pa}[1]{\left\lvert #1 \right\rvert} 
\newcommand{\pv}[1]{\left\langle #1 \right\rangle} 
\newcommand{\ps}[1]{\left\{ #1 \right\}} 
\newcommand{\pp}[1]{\left( #1 \right)} 
\newcommand{\inv}{^{-1}}
\newcommand{\Z}{\mathbb{Z}}

\begin{document}
\begin{figure}
\vspace*{-62pt}
\vspace{.75in}
\end{figure}
\vspace*{-45pt}
\leftline{\smalltt\#A1} \vskip -12.5pt
\centerline{\smalltt  INTEGERS 26 (2026)}
\vskip 3pt \hrule																													
\begin{center}
\uppercase{\bf \boldmath Comparing Left and Right Quotient Sets in Groups}
\vskip 20pt
{\bf Julian Duvivier}\\
{\smallit Department of Mathematics, Reed College, Portland, OR}\\ 
{\tt june@duvivier.us}\\ 
\vskip 10pt

{\bf Xiaoyao Huang}\\
{\smallit Department of Mathematics, University of Michigan, Ann Arbor, MI}\\ 
{\tt xyrushac@umich.edu}\\ 
\vskip 10pt
{\bf Ava Kennon}\\
{\smallit Department of Mathematics, Amherst College, Amherst, MA}\\ 
{\tt akennon26@amherst.edu}\\ 
\vskip 10pt
{\bf Say-Yeon Kwon}\\
{\smallit Department of Mathematics, Princeton University, Princeton, NJ}\\ 
{\tt sk9017@princeton.edu}\\
\vskip 10pt
{\bf Steven J. Miller}\\
{\smallit Department of Mathematics, Williams College, Williamstown, MA}\\ 
{\tt sjm1@williams.edu}\\
\vskip 10pt
{\bf Arman Rysmakhanov}\\
{\smallit Department of Mathematics, Williams College, Williamstown, MA}\\ 
{\tt ar21@williams.edu}\\
\vskip 10pt
{\bf Pramana Saldin}\\
{\smallit Department of Mathematics, University of Wisconsin, Madison, WI}\\ 
{\tt saldin@wisc.edu}\\
\vskip 10pt
{\bf Ren Watson}\\
{\smallit Department of Mathematics, University of Texas at Austin, Austin, TX}\\ 
{\tt renwatson@utexas.edu}\\ 
\end{center}
\vskip 20pt
\centerline{{\smallit Revised: }April 9, 2026} 
\vskip 30pt

\centerline{\bf Abstract}
\noindent
For a finite subset $A$ of a group $G$, we define the right quotient set and the left quotient set of $A$ as
\begin{align*}
    AA^{-1}& \ ~:=~ \ \{a_1a_2^{-1}:a_1,a_2\in A\},\\
    A^{-1}A&\ ~:=~\ \{a_1^{-1}a_2:a_1,a_2\in A\}
\end{align*}
respectively. While the right and left quotient sets are equal if $G$ is abelian, subtleties arise when $G$ is a non-abelian group, where the cardinality difference $|AA^{-1}| - |A^{-1}A|$ may take on arbitrarily large values. Using the results of Martin and O'Bryant on the cardinality differences of sum sets and difference sets in $\mathbb{Z}$, we prove that in the infinite dihedral group, $D_\infty \cong \mathbb{Z} \rtimes \mathbb{Z}/2\mathbb{Z}$, every integer difference is achievable. Further, we prove that in $F_2$, the free group on \(2\) generators, an integer difference is achievable if and only if that integer is even, and we explicitly construct subsets of $F_2$ that achieve every even integer. We further determine the minimum cardinality of $A \subset G$ so that the difference between the cardinalities of the left and right quotient sets is nonzero, depending on the existence of order $2$ elements in $G$. To prove these results, we construct difference graphs $D_A$ and $D_{A^{-1}}$ which encode equality in the right and left quotient sets respectively. We observe a bijection from edges in $D_A$ to edges in $D_{A^{-1}}$ and count connected components in order to obtain our results on cardinality differences $|AA^{-1}| - |A^{-1}A|$.

\pagestyle{myheadings}
\markright{\smalltt INTEGERS: 26 (2026)\hfill}
\thispagestyle{empty}
\baselineskip=12.875pt
\vskip 30pt

\section{Introduction}

\noindent This paper is dedicated with thanks to Carl Pomerance and Mel Nathanson. The genesis for this work came from conversations the fifth named author had at the Integers Conference in Georgia in 2025 on \emph{Results in Additive \& Elementary Number Theory Inspired by Carl and Mel}. His presentation described how much of his mentoring of students has been influenced by each, in particular problems in the orbit of MSTD sets (from Mel) and multiplicative and quotient structures (from Carl). This led to springboard problems for the SMALL 2025 REU, which led to the work below.

\subsection{Background}

Given a subset $A$ of $[N] ~\coloneq~ \{1,...,N\}$, the \textit{sumset} and \textit{difference sets} are defined as
    \begin{align*}
    A+A&\ ~\coloneqq~ \ \{a_1+a_2:a_1,a_2 \in A\},\\A-A&\ ~\coloneqq~ \ \{a_1-a_2:a_1,a_2\in A\}
    \end{align*}
respectively. A natural comparison arises between the cardinalities of the sum and difference sets. Our set $A$ is said to be {\em sum dominated or MSTD} (more sums than differences) if $|A+A|>|A-A|$ and {\em difference dominated (MDTS)} if $|A-A| > |A +A|$. One might expect that the difference set would have a greater cardinality as addition is commutative in $\mathbb Z$ while subtraction is not. In particular, if we consider a pair of distinct elements \( a_1, a_2 \in A\), \(a_1 - a_2\) and \(a_2-a_1\) are distinct elements in the difference set, while \(a_1+a_2=a_2+a_1\) is a single element in the sumset. However, it is possible to have a set that is sum dominated. The earliest examples of MSTD sets were discovered in the 1960s by Conway (\( \{0,2,3,4,7,11,12,14\}\)).
Surprisingly, against the intuition that MSTD sets should form a vanishing proportion of subsets of $[N]$ as $N$ grows large, in 2006 Martin and O'Bryant \cite{martin2006setssumsdifferences} proved that the proportion of subsets of $[N]\subseteq\Z$ which are MSTD does not vanish as $N\to\infty$. Since then, extensive research has expanded the classical MSTD problem to various settings, including higher dimensions and various  families \cite{kim2020constructionsgeneralizedmstdsets, chu2019generalizationscuriousfamilymstd, chu2019infinitefamiliespartitionsmstd}.

A natural extension of the study of MSTD sets among subsets of $\mathbb Z$ is to ask the question for general groups. This has been investigated in previous work by \cite{ascoli2022sum,miller2014subsetsbalancedfinitegroups,nathanson2006setssumsdifferences,zhao2010counting}.
In this paper, the main groups of interest are the free group \(F_2\) on two generators and the infinite dihedral group $D_\infty\cong\Z\rtimes\Z/2\Z$.

\begin{definition}
    Let $A$ be a set of generators with \(|A|=n\).
    The \textit{free group on \(A\)}, denoted $F(A)$ or $F_n$, is the group consisting of all reduced words over the alphabet $A\cup A\inv = \ps{x : x\in A}\cup \ps{x^{-1} : x\in A}$ under the operation of concatenation followed by reduction, where reduction means canceling adjacent inverse pairs.
\end{definition}
We now define sumsets and difference sets in general groups. Let $A \subset G$ be a finite subset of a group $G$ equipped with the operation $\times:(x,y)\mapsto xy$ and the inverse map $*^{-1}:x\mapsto x^{-1}$.
Although some authors use the terminology sumset and difference set in this context \cite{ascoli2022sum,miller2014subsetsbalancedfinitegroups,nathanson2006setssumsdifferences,zhao2010counting}, we refer to the {\em product set} and the {\em quotient set} of $A$, which are defined as
\begin{align*}
   AA &\ ~\coloneqq~ \  \{xy:x,y\in A\},
    \\A  A^{-1} &\ ~\coloneqq~ \  \{xy^{-1}:x,y\in A\}
\end{align*}
respectively.

More generally, let \(G\) be any group.
Let \(A = \ps{a_1,\dots,a_n} \subseteq G\). Define
\(A\inv \coloneqq \ps{a_1\inv,\dots, a_n\inv}\). We consider the {\em right quotient set} and {\em left quotient set} of $A$, defined by 
\begin{align*}
    AA\inv \ ~\coloneqq~ \  \ps{a_i\cdot a_j\inv : a_i,a_j\in A}, \\
    A\inv A \ ~\coloneqq~ \  \ps{a_i\inv\cdot a_j : a_i,a_j\in A}
\end{align*}
respectively.
In the classical MSTD problem, one natural area of study has been the cardinality of the smallest MSTD set. Hegarty showed that the smallest such example in \(\Z\) has cardinality \(8\), and
is unique up to affine translation \cite{Hegarty_2007}. Another natural question asks for the set of values that \(|A+A|-|A-A|\) can attain across all $A \subset [N]$.
Martin and O'Bryant proved, over finite subsets of \(\Z\), that this difference achieves all integers
\cite{martin2006setssumsdifferences}.

Inspired by this, we consider similar questions about left and right quotient sets. First note that
the problem of constructing a set \(A\) where
\(|AA\inv| > |A\inv A|\) is the same as the problem of \(|AA\inv| < |A\inv A|\)
(by replacing \(A\) with \(A\inv\)).
Therefore, a more natural question to ask is ``What is the smallest \(A\)
where the left and right quotient sets are not equal?''
Inspired by \cite{martin2006setssumsdifferences}, we give results on what values \(|AA\inv| - |A\inv A|\) can take
as \(A\) ranges over finite subsets of a group \(G\).

There has been work generalizing MSTD questions to abelian groups \cite{miller2014subsetsbalancedfinitegroups, nathanson2006setssumsdifferences,zhao2010counting}. However, our questions are only relevant for non-abelian groups since $AA^{-1}=A^{-1}A$ for $A$ a subset of an abelian group.

\subsection{Notation and Main Results}
\begin{theorem}\label{theor:group-difference-is-even}
    Let \(G\) be a group with no elements of order \(2\).\footnote{While we use the language of ``a group with no elements of order \(2\)'', we remark that Tao refers to such a group as a ``a \(2\)-torsion-free group'' in \cite{tao2011productsetestimatesnoncommutative}.}
    Let \(A \subseteq G\) be a finite subset. Then
    \(|AA\inv| - |A\inv A|\) is even.
\end{theorem}
This result does not hold in the context of groups with elements of order \(2\). In fact,
we can construct examples where odd values of \(|AA\inv| - |A\inv A|\)
are achieved.
\begin{example}\label{ex:infinite-dihedral}
    Let \(G = D_{\infty} \coloneqq \pv{r,s \mid sr = r\inv s, s^2=e}\) be the
    infinite dihedral group.
    For every \(n\in \Z\), there exists a subset $A_n \subseteq D_\infty$ such that $|A_nA_n^{-1}|-|A_n^{-1}A_n|=n$. Indeed, let \(B \subseteq \mathbb{Z}\)
    be a finite subset to be determined later.
    Consider the subset
    \[
        A \ ~\coloneqq~ \  \{r^b : b\in B\} ~\cup~ \{sr^b : b\in B\} ~\subseteq~ D_\infty.
    \]
    We see that
    \begin{align*}
        AA\inv \ ~=~ \
         &\underbrace{\ps{r^{b - b'} : b,b'\in B}}_{A_1} ~\sqcup~ \underbrace{\ps{sr^{b'-b} : b,b'\in B}}_{A_2}, \\
        A\inv A \ ~=~ \
         &\underbrace{\ps{r^{b'-b} : b,b'\in B}}_{A_1'} ~\sqcup~ \underbrace{\ps{sr^{b'+b} : b,b'\in B}}_{A_2'}. \\
    \end{align*}
    Notice that \(A_1=A_1'\), and \(|A_2| = |B-B|\) and \(|A_2'| = |B+B|\).
    Therefore,
    \[
        |AA\inv| - |A\inv A|\  ~=~\ |B-B| - |B+B|.
    \]
    By \cite[Theorem 4]{martin2006setssumsdifferences}, the latter difference
    ranges over every integer.
\end{example}

Considering in particular the free group \(G=F_2\), we can construct
examples where every even integer is achieved.
\begin{theorem}[\textit{\(F_2\) achieves all even possible differences}]
    \label{theor:f2-has-all-even-differences}
     For all \(n\in \mathbb{Z}\), there exists a set \(A_n \subseteq F_2\) such that
    \(|A_nA_n\inv| - |A_n\inv A_n|  \ ~=~\ 2n\).
\end{theorem}

The construction for the previous theorem uses a subset \(A \subseteq F_2\)
of cardinality \(5\) that satisfies \(|AA\inv| \neq |A\inv A|\).
The following theorem shows this construction is optimal
with respect to the size of \(A\) for a more general class of groups.

\begin{theorem}\label{theor:size-of-A}
     Let \(G\) be a group. Let $A \subseteq G$ be a finite subset
     and suppose that $|AA^{-1}| \neq |A^{-1}A|$. Then
     \begin{itemize}
         \item without any further assumptions, \(|A| \ge 4\), and

         \item if \(G\) is a group with no elements of order
         \(2\), then \(|A| \ge 5\).
     \end{itemize}
\end{theorem}

A brute force search among groups of small order
shows that the bound \(|A|\ge 4\) is sharp (see Example \ref{ex:qdih-group-A-4}).
The main tool we use to prove this result is a graph associated to
the left (and right) quotient sets, which we call the difference graph.
The cardinalities of \(AA\inv\) and \(A\inv A\) can be interpreted as
the number of connected components on these graphs. Making use of a bijection of
edges between these graphs, we perform an argument based on the properties of
this graph to prove that when \(|A| \le 3\) (resp.\ \(|A| \le 4\) when \(A\)
has no elements of order \(2\)), the number of connected components does not change under the bijection of edges.

\section{Left vs. Right Quotient Sets}

    \subsection{Graph Construction}
    To prove our results, we define the \textit{difference graph} $D_A$ of a finite subset  \(A \coloneqq \{a_1,a_2,...,a_n\} \subseteq G\). The graph \(D_A=(V,E)\) is defined as follows:
    \begin{enumerate}
        \item The vertex set is given by \(V \coloneqq[n]\times [n]\).
        \item The edge set \(E(D_A)\) is given by the relation
    \[(i,j)\sim (k,\ell) \text{ if and only if } a_ia_j\inv = a_ka_\ell\inv.\]
    \end{enumerate}
    The difference graph is directed and \textit{not} simple (we allow self-loops).

    Similarly, for \(D_{A^{-1}}\), we have the edge relation
    \[(i,j)\sim(k,\ell)\text{ if and only if }a_i^{-1}a_j=a_k^{-1}a_\ell.\]
    We first note the following basic facts about \(D_A\).

\begin{lemma}[Properties of $D_A$]\label{lem:properties-of-D_A}\
Let \(i,j,k,\ell\in [n]\).
\begin{enumerate}
    \item \([(i,j), (k,\ell)] \in E(D_A) \) if and only if \( [(j,i), (\ell,k)] \in E(D_A)\).
\item The following types of edges are not present in \(E(D_A)\).

\begin{enumerate}
\item \([(i,j), (k,k)]\), an edge connecting to the diagonal, provided that \(i\neq j\). \label{lemitem:no-diagonal}
\item \([(i,j), (i,k)]\) (or \([(j,i), (k,i)]\), but this is handled by (1)), an edge connecting vertices on the same axis. \label{lemitem:no-same-axis}
\item If \(G\) has no elements of order \(2\), then no edge connects a vertex $(i,j)$ to its symmetric pair $(j,i)$, provided that \(j\neq i\). \label{lemitem:no-symmetric-connections}
\end{enumerate}
\item \([(i,i),(j,j)]\in E(D_A)\). \label{lemitem:all-diagonals}
\item If \(C\) is a connected component in \(D_A\), then \(C\) is a clique. \label{lemitem:component-cliques}
    \end{enumerate}
\end{lemma}

\begin{proof}
(1) Suppose \(a_ia_j\inv = a_ka_\ell\inv\). Then
\[
    a_ja_i\inv\ ~=~\ (a_ia_j\inv)\inv \ ~=~ \  (a_ka_\ell\inv)\inv \ ~=~ \  a_\ell a_k\inv.
\]
Hence, \([(j,i), (\ell,k)] \in E(D_A)\) and note the reverse follows.

(2a) Let \([(i,j), (k,k)] \in E(D_A)\) where \(i\neq j\). Thus
\[
    a_ia_j\inv \ ~=~ \  a_ka_k\inv \ ~=~ \  e,
\]
which implies \(a_i~=~a_j\). But \(i\neq j\), a contradiction.

(2b) Without loss of generality, take the edge \([(i,j), (i,k)]\). Then we have
    \[
        a_ia_j\inv \ ~=~ \  a_ia_k\inv,
    \]
    which implies \(a_j ~=~ a_k\).

    (2c)  Let \([(i,j), (j,i)] \in E(D_A)\) where \(i\neq j\). Then,
    \[
        a_ia_j\inv \ ~=~ \  a_ja_i\inv \ ~=~ \  (a_ia_j\inv)\inv,
    \]
    which implies
    \[
     (a_ia_j\inv)^2 \ ~=~ \  e.
    \]
    Because $G$ has no elements of order $2$, so $a_ia_j^{-1}=e$ and \(a_i~=~a_j\),
    contradicting our assumption that \(i\neq j\).

    (3) Consider \(a_ia_i\inv ~=~ e ~=~ a_ja_j\inv\), thus \([(i,i),(j,j)]\in E(D_A)\).

    (4) This follows because equality is an equivalence relation.
\end{proof}

\begin{remark}
    In light of Property (4) Lemma \ref{lem:properties-of-D_A}, we refer to a connected component $C$ in $D_A$ consisting of $k$ elements simply by indicating its vertices $C=(a_1,b_1)(a_2,b_2)\cdots (a_k,b_k)$.
\end{remark}

\begin{remark}\label{rmk:transpose-definition}
    Property (1) is the same as saying the ``transpose'' operation
    \begin{align*}
        {T}\colon{E(D_A)}&\to{E(D_A)} \\
        [(i,j),(k,\ell)] &\mapsto [(j,i),(\ell,k)]
    \end{align*}
    is a graph automorphism.
\end{remark}

Let \(\mathcal{C}(D_A)\) be the set of connected components of
\(D_A\) and \(c(D_A) = \pa{\mathcal{C}(D_A)}\) be the number of
connected components. Note that \(c(D_A) = |AA\inv|\) and \(c(D_{A\inv}) = |A\inv A|\).

Using the fact that
\begin{align*}
    a_ia_j^{-1}\ ~=~ \ a_ka_{\ell}^{-1} \text{\ \ exactly when\ \ } a_k^{-1}a_{i} \ ~=~ \  a_{\ell}^{-1}a_j,
\end{align*}
we obtain a bijection of edges between $D_A$ and $D_{A^{-1}}$ as follows:
  \begin{align}
    \begin{aligned}
    {\phi}\colon{E(D_A)}&\to{E(D_{A^{-1}})} \\
    [(i,j), (k,\ell)]&\mapsto [(k,i), (\ell, j)].\nonumber
    \end{aligned}
\end{align}
A priori, this map is only well-defined if we consider edges as directed
(the reverse edge gets mapped to the transpose of the original edge) and
allow loops (they get mapped to the diagonal). However, since
$T$ is an automorphism of $D_A$, we may take $D_A$ to be
undirected. This is a well-motivated map: its existence already tells us that the \emph{additive energies} \(\Lambda(A,A\inv) ~\coloneqq~ \{\#(a_1,a_2,a_3,a_4) \in A^4 : a_1a_2^{-1}=a_3a_4^{-1}\}\) and \(\Lambda(A\inv, A) ~\coloneqq~ \{\#(a_1,a_2,a_3,a_4) \in A^4 : a_1^{-1}a_2=a_3^{-1}a_4\}\) are equal in non-commutative groups (see \cite{tao2011productsetestimatesnoncommutative}).

\begin{example}
When \(n=5\), consider the example in Figure \ref{fig:mltr-graph-example}.
\begin{figure}[h]
    \begin{tikzpicture}[scale=1.0]
    \foreach \x in {0,...,4}{
        \foreach \y in {0,...,4}{
            \node (\x\y) at (\x,\y) {};
            \filldraw (\x\y) circle [fill=black,radius=2pt] {};
        }
    }

    \foreach \x in {5,...,9}{
        \foreach \y in {0,...,4}{
            \node (\x\y) at ({(\x + 2)},\y) {};
            \filldraw (\x\y) circle [fill=black,radius=2pt] {};
        }
    }

    \foreach \x in {1,...,5}{
        \node [below=10pt] at ({(\x-1)},0) {\(a_{\x}\)};
    }

    \foreach \x in {1,...,5}{
        \node [below=10pt] at ({(\x+6)},0) {\(a_{\x}\)};
    }

    \foreach \y in {1,...,5}{
        \node [left=5pt] at (0,{(\y-1)}) {\(a_{\y}^{-1}\)};
    }

    \foreach \y in {1,...,5}{
        \node [left=5pt] at (7,{(\y-1)}) {\(a_{\y}^{-1}\)};
    }

    \foreach \x in {0,...,3} {
        \draw [-] (\x,\x) -- ({\x+1},{\x+1});
        \draw [-] (\x+7,\x) -- ({\x+8},{\x+1});
    }

    \draw [-] (0,2) -- (1,4) -- (2,3) -- cycle;
    \draw [-] (2,0) -- (4,1) -- (3,2) -- cycle;

    \draw [-] (0+7,2) -- (2+7,3);
    \draw [-] (0+7,1) -- (2+7,4);
    \draw [-] (1+7,2) -- (4+7,3);

    \draw [-] (2+7,0) -- (3+7,2);
    \draw [-] (1+7,0) -- (4+7,2);
    \draw [-] (2+7,1) -- (3+7,4);

    \draw[->] (4.5,2) -- (5.8,2)
        node [above, text centered, midway] {\(\phi\)};
\end{tikzpicture}
    \caption{This example shows that the number of connected components in \(D_A\) changes after applying the map \(\phi\).}
    \label{fig:mltr-graph-example}
\end{figure}

The former (left) graph prior to the mapping $\phi$ has \(17\) connected components, and the latter (right) graph after $\phi$ is applied has
\(15\) connected components.

The first graph corresponds to
\[
    a_3a_1\inv \ ~=~ \  a_4a_3\inv \ ~=~ \  a_5a_2\inv
\]
being satisfied, and no other relations between words (other than
the diagonal). The set
\begin{equation}\label{eq:f2-set-with-lr-difference}
        A \ ~\coloneqq~\ \ps{x, xz, y\inv, y\inv x\inv y\inv, y\inv z} \subseteq F_3\ ~=~ \ F(\ps{x,y,z})
\end{equation}
satisfies this property, therefore giving us an example where \(|AA\inv| \neq |A\inv A|\).
\end{example}

\subsection{Possible Differences}
The natural question that arises is: what are the possible differences between the cardinalities of the right and left quotient sets of $A \subset G$? In Example \ref{ex:infinite-dihedral} we gave a construction on the infinite dihedral group $D_\infty$ demonstrating the difference \(|AA\inv|-|A\inv A|\) achieves every possible \(n \in \Z\). Given the restriction that there are no elements of order 2 in our group, we can conclude that \(|AA\inv |-|A\inv A|\) is even. The proof of Theorem \ref{theor:group-difference-is-even} follows.

\begin{proof}[Proof of Theorem \ref{theor:group-difference-is-even}]
      It suffices to show that \(c(D_A) - c(D_{A\inv})\) is even.
    We claim that \(c(D_A)\) is odd.
    We divide \(\mathcal{C}(D_A)\) into two disjoint classes. \begin{enumerate}
        \item The connected components that are fixed under $T$ (i.e., $T(C) = C$), and
        \item the connected components which are disjoint from
    their image under \(T\); that is, those connected components that are swapped with a distinct component under \(T\).
    \end{enumerate}
    Denote these sets by
    \(\mathcal{C}_1\) and \(\mathcal{C}_2\) respectively.
    Note that \(\pa{\mathcal{C}_2}\) is even as each component comes in pairs.

    We claim that \(\mathcal{C}_1\) contains only the diagonal
    (the diagonal is a connected component by Lemma \ref{lem:properties-of-D_A} (\ref{lemitem:all-diagonals}, \ref{lemitem:no-diagonal})). Indeed, if any other component
    \(C\) belongs to \(\mathcal{C}_1\), then there exists a vertex \((i,j)\)
    with \(i\neq j\) such that \([(i,j),(j,i)]\in C\), which contradicts Lemma \ref{lem:properties-of-D_A} \eqref{lemitem:no-symmetric-connections}. Therefore, there is exactly one component of \(\mathcal{C}_1\) while the rest of the connected components come in symmetric pairs, so \(|\mathcal{C}_1|\) is odd. Applying the same reasoning to \(c(D_{A^{-1}})\) shows  it is odd, so \(c(D_A) - c(D_{A\inv})\) is even.
\end{proof}

    We now consider the set of possible differences in $F_2$. In order to prove Theorem \ref{theor:f2-has-all-even-differences}, we use the following.

\begin{fact}\label{fact:f_m-injto-f_2}
    Let \(m\ge 2\) be an integer.
    Then there exists an embedding \(F_m \hookrightarrow F_2\).
\end{fact}

\begin{proof}
    See, for example \cite[Chapter I. Proposition 3.1.]{LyndonSchuppCombinatorial}.
\end{proof}

Writing out this embedding explicitly allows us to describe a subset of \(F_2\) where \(|AA\inv| \neq |A\inv A|\). Indeed, take the set $A$ given in \eqref{eq:f2-set-with-lr-difference} with the embedding \(F_3\hookrightarrow F_2\) by
\(x\mapsto x^2\), \(y\mapsto xy\), \(z\mapsto xy^{-1}\) to get the set
\[
    \{x^2,x^3y\inv, y\inv x\inv,y\inv x^{-3}y\inv x \inv, y^{-2}\} \subseteq F_2.
\]


\begin{proof}[Proof of Theorem \ref{theor:f2-has-all-even-differences}]
    Any set \(A \subseteq \pv{a} \cong \mathbb{Z}\) with \(a\in F_2\)
    yields the \(n=0\) case.
    By replacing \(A\) with \(A\inv\), it suffices to prove the result for
    positive \(n\). So, we construct a family of sets \(A_n \subseteq F_m\) for some \(m\ge 2\) and then compose it with the embedding in Fact \ref{fact:f_m-injto-f_2}.
    The following set (described above) in \(F_3=F(\ps{x,y,z})\) for $n=1$
    \[
    \begin{aligned}
      &A = \{x, \ xz, \ y^{-1}, \ y^{-1}x^{-1}y^{-1}, \ y^{-1}z\},\\
      &A^{-1} = \{x^{-1}, \ z^{-1}x^{-1}, \ y, \ yxy, \ z^{-1}y\},
    \end{aligned}
\]
    has
\[
\begin{aligned}
AA^{-1}\ ~=~ \ \{&e,\ xz^{-1}x^{-1},\ xy,\ xyxy,\ xz^{-1}y,\
xzx^{-1},\ xzy,\ xzyxy,
\ y^{-1}x^{-1},\ y^{-1}z^{-1}x^{-1},\ \\&y^{-1}z^{-1}y,\
y^{-1}x^{-1}y^{-1}x^{-1},\ y^{-1}x^{-1}y^{-1}z^{-1}x^{-1},\ y^{-1}x^{-1}y^{-1}z^{-1}y,\\
&y^{-1}zx^{-1},\ y^{-1}zy,\ y^{-1}zyxy\},
\end{aligned}
\]
\[
\begin{aligned}
A^{-1}A\ ~=~ \ \{&e,\ z,\ x^{-1}y^{-1},\ x^{-1}y^{-1}x^{-1}y^{-1}, x^{-1}y^{-1}z,\ z^{-1},\ z^{-1}x^{-1}y^{-1}, z^{-1}x^{-1}y^{-1}x^{-1}y^{-1},\  \\ &z^{-1}x^{-1}y^{-1}z,\
 yx,\ yxz,\
yxyx,\  yxyxz,\
 z^{-1}yx, \ z^{-1}yxz\}.
\end{aligned}
\]
    Hence,
    \[
        |AA\inv| - |A\inv A| \ ~=~ \  17 - 15 \ ~=~ \  2.
    \]

    More generally for \(n\ge 1\), \(A_n\) is constructed as a subset of
    \(F_{3n} = F(\ps{x_1,y_1,z_1,\dots,x_n,y_n,z_n})\) as follows. Let
    \[
        A_n \ ~\coloneqq~ \  \bigcup_{i=1}^{n}\ps{x_i, y_i\inv, y_i\inv x_i\inv y_i\inv, x_iz_i, y_i\inv z_i}.
    \]
    We claim
    \[
        |A_nA_n\inv| - |A_n\inv A_n| \ ~=~ \  2n.
    \]
    Define \(A_n^{(i)} = \ps{x_i, y_i\inv, y_i\inv x_i\inv y_i\inv, x_iz_i, y_i\inv z_i}\) for \(1\le i \le n\). We have
    \begin{align*}
        A_nA_n\inv &\ ~=~ \  \bigsqcup_{i=1}^{n}\bigsqcup_{j=1}^{n} A_n^{(i)}\pp{A_n^{(j)}}^{-1},\\
        A_n\inv A_n &\ ~=~ \  \bigsqcup_{i=1}^{n}\bigsqcup_{j=1}^{n} \pp{A_n^{(i)}}^{-1}A_n^{(j)}.
    \end{align*}
    This implies\footnote{
    For \(i \ne j\), the sets \(A_n^{(i)}\) and \(A_n^{(j)}\) are supported on disjoint sets of generators in \(F_{3k}\), meaning any product of the form \(ab^{-1}\) or \(a^{-1}b\) where $a \in A_n^{(i)} $ and \(b\in A_n^{(j)}\) produces a word involving letters from distinct alphabets. Thus, all such products are distinct. However, if \(i=j\), the words produced come from the same alphabet, which do not generate distinct products. } 
    \begin{align*}
        |A_nA_n\inv| &\ ~=~ \  \sum_{i=1}^{n}\sum_{j=1}^{n} \pa{A_n^{(i)}\pp{A_n^{(j)}}^{-1}} \\
        &\ ~=~ \  \sum_{i=1}^{n} \pa{A_n^{(i)}\pp{A_n^{(i)}}^{-1}} + \sum_{i=1}^{n}\sum_{\substack{j=1 \\ j\neq i}}^{n}\pa{A_n^{(i)}\pp{A_n^{(j)}}^{-1}} \\
        &\ ~=~ \  n\cdot |AA\inv| + n\cdot(n-1)\cdot |A|^2,
    \end{align*}
    and
    \begin{align*}
        |A_n\inv A_n| &\ ~=~ \  \sum_{i=1}^{n}\sum_{j=1}^{n} \pa{\pp{A_n^{(i)}}^{-1}A_n^{(j)}} \\
        &\ ~=~ \  \sum_{i=1}^{n} \pa{\pp{A_n^{(i)}}^{-1}A_n^{(i)}} + \sum_{i=1}^{n}\sum_{\substack{j=1 \\ j\neq i}}^{n}\pa{\pp{A_n^{(i)}}^{-1}A_n^{(j)}} \\
        &\ ~=~ \  n\cdot |A\inv A| + n\cdot(n-1)\cdot |A|^2.
    \end{align*}
    Therefore,
    \[
        |A_nA_n\inv| - |A_n\inv A_n| \ ~=~ \  n\cdot (|AA\inv| - |A\inv A|) \ ~=~ \  2n. \qedhere
    \]
\end{proof}

Theorem \ref{theor:group-difference-is-even} and Theorem \ref{theor:f2-has-all-even-differences} establish that every even integer arises in the difference between the left and right quotient set. Naturally, we may now ask how large such differences can be in terms of the cardinality of $A$. Define
\[
    M_n \ ~\coloneqq~ \  \sup_{\substack{A \subseteq F_2 \\ |A| = n}}\left||AA^{-1}| - |A^{-1}A|\right|.
\]
\begin{proposition}\label{prop:growth-of-lr-difference} We have
$M_n = \Theta(n^2)$.
\end{proposition}

\begin{proof}
    We establish both an upper and lower bound. For the lower bound, observe that for any finite subset $A \subseteq F_2$ with $|A| = n$ it follows that $|AA\inv| \le n^2$ and $|A\inv A| \le n^2$. This implies $||AA\inv|-|A\inv A||\le n^2$ thus $M_n \le n^2$ so $M_n = O(n^2)$. For the lower bound, let \begin{align*}
        A_k &\ ~\coloneqq~ \  \{x^i: 1 \le i \le k\},\\
        B_k&\ ~\coloneqq~ \ \{x^iy:1\le i \le k\}, \\
        C_k&\ ~\coloneqq~ \  A_k\cup B_k.
    \end{align*} Note that $|C_k| = 2k $ and thus let $n = 2k$. Next, we compute the size of the right quotient set:
    \begin{align*}
        C_kC_{k\inv} &\ ~=~ \ (A_k\cup B_k)(B_{k}\inv \cup A_{k}\inv)\\
        &\ ~=~ \ (A_kB_{k}\inv)\cup(B_kB_{k}\inv)\cup(A_kA_{k}\inv)\cup(B_kA_{k}\inv).
    \end{align*}
    As $A_kA_{k}\inv = \{x^{i-j}\}=B_kB_{k}\inv$, this accounts for $2k-1$ elements. As $A_kB_k\inv = \{x^iy\inv x^{-j}\}$, this adds $k^2$ elements, and lastly for $B_kA_k\inv=\{x^iyx^{-j}\}$, there are also $k^2$ elements. Thus,
    \begin{align*}
        |C_kC_k\inv|\ ~=~ \ 2k^2+2k-1.
    \end{align*}
    For the left quotient set, we obtain
    \begin{align*}
        C_k\inv C_k&\ ~=~ \ (B_k\inv \cup A_k\inv)(A_k \cup B_k)\\
        &\ ~=~ \ (B_k\inv A_k)\cup(B_k\inv B_k)\cup(A_k\inv A_k)\cup(A_k\inv B_k).
    \end{align*}
    As each of these four terms contribute $2k-1$ elements,
    \begin{align*}
        |C_k\inv C_k|\ ~=~ \ 4(2k-1)=8k-4.
    \end{align*}
    Then, taking the difference between the left and right quotient set yields
    \begin{align*}
        ||C_kC_k\inv|-|C_k\inv C_k|| &\ ~=~ \  (2k^2+2k-1)-(8k-4)\\
        &\ ~=~ \ 2k^2-6k+1.
    \end{align*}
    Substituting $k=n/2$ gives $M_n\geq \frac{1}{2}n^2 -3n +1 = \Omega(n^2)$. Thus, $M_n=\Theta(n^2).$
\end{proof}

\begin{remark}
    The following remark is mentioned in \cite[Remark 4.4.]{tao2011productsetestimatesnoncommutative}. If
    \(H \subseteq G\) is a finite subgroup of some group \(G\) and
    \(g\) is an element of \(G\) not in the normalizer of \(H\),
    then the set \(A \coloneqq Hg\cup H\) has \(A\inv A\) about
    the same size as \(H\), but \(AA\inv\) can be possibly large. If we slightly modify this construction, letting \(H\) be a \textit{subset} of \(G\), then the proof of
    Proposition \ref{prop:growth-of-lr-difference} is the case
    \(H = \ps{x,x^2,\dots,x^k}\) and \(g=y\). Indeed,
    \(|A^{-1}A| = O(|H|)\), but \(|AA^{-1}| = O(|H|^2)\).
\end{remark}

\subsection{Cardinality of $A$}
The set $A$ given in
\eqref{eq:f2-set-with-lr-difference} is extremal in the sense that it is a subset of $F_2$ with minimal cardinality such that $|AA^{-1}|\neq |A^{-1}A|$.
In fact, it has minimal cardinality among subsets with $|AA^{-1}|\neq |A^{-1}A|$ in any group with no elements of order $2$ as a consequence of Theorem \ref{theor:size-of-A}.
\begin{lemma}\label{lem:clique-impossible}
Let $G$ be a group with $A \subset G$ and \(|A| = n\). Then \(D_A\) contains no connected component with more than \(n\) elements.
\end{lemma}

\begin{proof}
    Suppose, for contradiction, \(C\) was a connected component with more than \( n\) elements.
    By the pigeonhole principle, two vertices in \(C\) have the same
    first coordinate. Since \(C\) is a clique by
    Lemma \ref{lem:properties-of-D_A} \eqref{lemitem:component-cliques}, this contradicts
    Lemma \ref{lem:properties-of-D_A} \eqref{lemitem:no-same-axis}.
\end{proof}

\begin{lemma}\label{lem:triangle-invariant}
    Let $G$ be a group with $A \subseteq G$ and $|A| = 4$. If $G$ does not have an element of order 2, then the number of connected components in $D_A$ is equal to the number of connected components in $D_{A^{-1}}$.
\end{lemma}

\begin{proof}
Suppose $D_A$ contains a $K_4$ clique. Then, up to relabeling and transposition, this clique has the form $(1,2)(2,3)(3,4)(4,1)$. Indeed, given a $K_4$ of the form $(1,j_1)(2,j_2)(3,j_3)(4,j_4)$, we may assume $j_1=2$ and $j_2=3$ up to relabeling. This forces $j_4=1$ and $j_3=4$. This clique is impossible if $G$ has no elements of order $2$ since $a_1a_3^{-1}=(a_1a_2^{-1})(a_2a_3^{-1})=(a_3a_4^{-1})(a_4a_1^{-1})=a_3a_1^{-1}$ and this contradicts Lemma \ref{lem:properties-of-D_A} \eqref{lemitem:no-symmetric-connections} given the assumption that $G$ has no elements of order $2$.

It remains to consider the case where the largest clique is a $K_3$, which we call a triangle. Consider a triangle $\triangle= (\alpha_1,\beta_1)(\alpha_2,\beta_2)(\alpha_3,\beta_3)$. By Lemma \ref{lem:properties-of-D_A} \eqref{lemitem:no-same-axis}, we know that the same number can appear at most twice in the coordinates \(\alpha_1\), \(\beta_1\), \(\alpha_2\), \(\beta_2\), \(\alpha_3\), and \(\beta_3\). Moreover the same number appears at most once for an \(\alpha_i\) and at most once for a \(\beta_j\), where \(1\le i,j\le 3\). Therefore, there are two cases to consider.\\

\noindent
\textbf{Case 1: There are 4 distinct elements in $\triangle$.}
Because of the condition that there is no element of order 2 in $G$, there cannot be an edge between $(i,j)$ and $(j,i).$ Therefore, up to relabeling, $\triangle=(1,2)(2,3)(3,4)$. By elementary group operations, the edge $[(1,3),(2,4)]$ and its transpose are also in $D_A$. Hence, $D_A$ and $D_{A^{-1}}$ look as in Figure \ref{fig:case 1.1}. Therefore, in this case, the number of connected components stays the same.
\begin{figure}[h]
\begin{center}
    \begin{tikzpicture}[scale=1.0]
    \foreach \x in {0,...,3}{
        \foreach \y in {0,...,3}{
            \node (\x\y) at (\x,\y) {};
            \filldraw (\x\y) circle [fill=black,radius=2pt] {};
        }
    }

    \foreach \x in {4,...,7}{
        \foreach \y in {0,...,3}{
            \node (\x\y) at ({(\x + 1)},\y) {};
            \filldraw (\x\y) circle [fill=black,radius=2pt] {};
        }
    }

    \foreach \x in {0,...,2} {
        \draw [-,dotted] (\x,\x) -- ({\x+1},{\x+1});
        \draw [-,dotted] (\x+5,\x) -- ({\x+6},{\x+1});
    }

    \draw [-] (0,1) to [bend left=20] (2,3);
    \draw [-] (0,1) -- (1,2) -- (2,3);
    \draw [-] (1,0) -- (2,1) -- (3,2);
    \draw [-] (1,0) to [bend right=20] (3,2);

    \draw [-] (0,2) -- (1,3);

    \draw[-] (2,0) -- (3,1);

    \draw [-] (0+5,1) to [bend left=20] (2+5,3);
    \draw [-] (0+5,1) -- (1+5,2) -- (2+5,3);
    \draw [-] (1+5,0) -- (2+5,1) -- (3+5,2);
    \draw [-] (1+5,0) to [bend right=20] (3+5,2);

    \draw [-] (0+5,2) -- (1+5,3);

    \draw[-] (2+5,0) -- (3+5,1);

    \draw[->] (3.5,1.5) -- (4.5,1.5)
        node [above, text centered, midway] {\(\phi\)};
\end{tikzpicture}
\end{center}
 \caption{Up to relabeling, $\triangle=(1,2)(2,3)(3,4)$ induces $D_A$ and $D_{A^{-1}}$ above.}
    \label{fig:case 1.1}
\end{figure}

 Additionally there may be a triangle $\triangle' = (3,1)(4,2)(1,4)$ and its
reflection under $T$. However, we can check that the image of this graph under $\phi$
is itself as in Figure \ref{fig:case1.2} below.
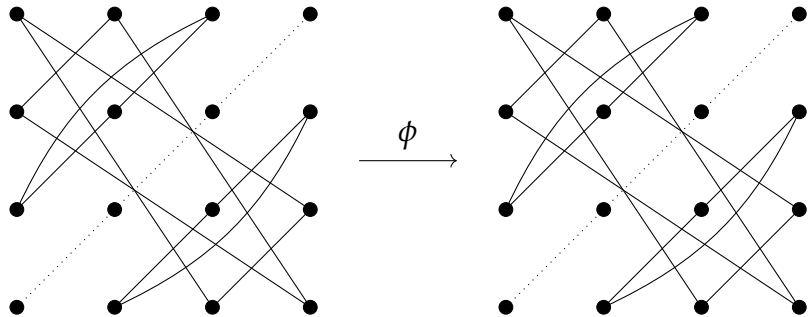
\begin{figure}[h]
\begin{center}
    \begin{tikzpicture}[scale=1.0]
    \foreach \x in {0,...,3}{
        \foreach \y in {0,...,3}{
            \node (\x\y) at (\x,\y) {};
            \filldraw (\x\y) circle [fill=black,radius=2pt] {};
        }
    }

    \foreach \x in {4,...,7}{
        \foreach \y in {0,...,3}{
            \node (\x\y) at ({(\x + 1)},\y) {};
            \filldraw (\x\y) circle [fill=black,radius=2pt] {};
        }
    }

    \foreach \x in {0,...,2} {
        \draw [-,dotted] (\x,\x) -- ({\x+1},{\x+1});
        \draw [-,dotted] (\x+5,\x) -- ({\x+6},{\x+1});
    }

    \draw[-] (0,1) to [bend left=20] (2,3);
    \draw[-] (0,1) -- (1,2) -- (2,3);
    \draw[-] (1,0) -- (2,1) -- (3,2);
    \draw[-] (1,0) to [bend right=20] (3,2);

    \draw [-] (0,2) -- (1,3) -- (3,0) -- cycle;

    \draw[-] (2,0) -- (3,1) -- (0,3) -- cycle;

    \draw[-] (0+5,1) to [bend left=20] (2+5,3);
    \draw[-] (0+5,1) -- (1+5,2) -- (2+5,3);
    \draw[-] (1+5,0) -- (2+5,1) -- (3+5,2);
    \draw[-] (1+5,0) to [bend right=20] (3+5,2);

    \draw[-] (0+5,2) -- (1+5,3) -- (3+5,0) -- cycle;

    \draw[-] (2+5,0) -- (3+5,1) -- (0+5,3) -- cycle;

    \draw[->] (3.5,1.5) -- (4.5,1.5)
        node [above, text centered, midway] {\(\phi\)};
\end{tikzpicture}
\end{center}
 \caption{ Up to relabeling, $\triangle'=(3,1)(4,2)(1,4)$ induces $D_A$ and $D_{A^{-1}} $ above.}
    \label{fig:case1.2}
\end{figure}

\vspace{8pt} \noindent
\textbf{Case 2: There are 3 distinct elements in $\triangle$.}
The only possibility in this case is that up to relabeling, $\triangle = (1,2)(2,3)(3,1).$ Notice that this triangle cannot be part of a $K_4$ because $(4,4)$ is not in the same connected component as $\triangle$ and all other coordinates will cause a contradiction with Lemma \ref{lem:properties-of-D_A} (\ref{lemitem:no-same-axis}).
Then, $D_A$ and $D_{A^{-1}}$ are as in Figure \ref{fig:case 2} below.

\begin{figure}[h]
    \centering
     \begin{tikzpicture}[scale=1.0]
    \foreach \x in {0,...,3}{
        \foreach \y in {0,...,3}{
            \node (\x\y) at (\x,\y) {};
            \filldraw (\x\y) circle [fill=black,radius=2pt] {};
        }
    }

    \foreach \x in {4,...,7}{
        \foreach \y in {0,...,3}{
            \node (\x\y) at ({(\x + 1)},\y) {};
            \filldraw (\x\y) circle [fill=black,radius=2pt] {};
        }
    }

    \foreach \x in {0,...,2} {
        \draw [-,dotted] (\x,\x) -- ({\x+1},{\x+1});
        \draw [-,dotted] (\x+5,\x) -- ({\x+6},{\x+1});
    }

    \draw[-] (0,1) -- (1,2) -- (2,0) -- cycle;
    \draw[-] (1,0) -- (2,1) -- (0,2) -- cycle;

    \draw[-] (0+5,1) -- (1+5,2) -- (2+5,0) -- cycle;
    \draw[-] (1+5,0) -- (2+5,1) -- (0+5,2) -- cycle;

    \draw[->] (3.5,1.5) -- (4.5,1.5)
        node [above, text centered, midway] {\(\phi\)};
\end{tikzpicture}
    \caption{Up to relabeling, $\triangle=(1,2)(2,3)(3,1)$ induces $D_A$ and $D_{A^{-1}}$ above.}
    \label{fig:case 2}
\end{figure}

\noindent
In the second case, we see it is impossible to place a second triangle
without two vertices being on the same axis, violating Lemma \ref{lem:properties-of-D_A} \eqref{lemitem:no-same-axis}.
\end{proof}

\begin{proof}[Proof of Theorem \ref{theor:size-of-A}]
    If \(|AA\inv| \neq |A\inv A|\), then the number of connected
    components in \(D_A\) is different from \(D_{A\inv}\).
    We claim at least one of \(D_A\) and \(D_A\inv\) contains a clique of size 3 or greater off the main diagonal (a non-diagonal clique). Indeed, since \(\phi\) is a bijection between the
    non-diagonal and self-loop edges,
    the number of connected components in \(D_A\) would be the same as
    \(D_{A\inv}\) if there were no non-diagonal connected components of size 3 or greater.

    Assume \(D_A\) has a non-diagonal clique of size 3 or greater.
    This is impossible if \(|A| \in \ps{1,2}\)
    by Lemma \ref{lem:clique-impossible}. If
    \(|A| = 3\), then the only possible triangles are
    \((1,2)(3,1) (2,3)\) and \((2,1)(1,3) (3,2)\). Since
    any \(D_A\) is invariant under transposition by \(T\), if \(D_A\) has
    a triangle, then it has both. Any other edge would either
    (1) connect the triangles to the diagonal, which is forbidden,
    or (2) connect the triangles to each other, which creates a component with more
    \(3\) vertices, contradicting Lemma \ref{lem:clique-impossible}. Thus, with $|A| = 3$, $D_A$ and $D_{A^{-1}}$ must have the same component counts. So, we have shown for any group $G$ if $|A| \le 3$ then $|AA\inv|=|A\inv A|$.

    Notice how the above proof does not use Lemma \ref{lem:properties-of-D_A} \eqref{lemitem:no-symmetric-connections}, so it holds for any group, yielding the first part of the theorem.

Next, suppose that $G$ has no elements of order $2$.
If $|A| = 4$, then Lemma \ref{lem:triangle-invariant} shows that $\phi$ preserves the number of connected components. Hence, we have shown if $G$ is a group with no elements of order 2 and $|A| \le 4$, then $|AA\inv| = |A\inv A|$.
\end{proof}

\begin{example}\label{ex:qdih-group-A-4}
If $G$ has elements of order $2$, then we may no longer use Lemma \ref{lem:properties-of-D_A} \eqref{lemitem:no-symmetric-connections} and proof of Lemma \ref{lem:triangle-invariant} fails. Indeed, we may check with Sage \cite{sagemath} that the set
\[
    A \ = \ \ps{(15)(26)(37)(48),\ (1256)(3874),\ (17)(35)(48),\ (18765432)} \ \subseteq \ S_8,
\]
which can be viewed as a subset of the quasidihedral group of order $16$,
has $|AA^{-1}| = 10$, but $|A^{-1}A| = 7$. This simultaneously shows
we can have an odd difference in the left and right quotient sets and also
    that $|A| = 4$ can achieve a nonzero difference.
\end{example}

\section{Future Work}

In a finite subset \(S\) of a group \(G\), one may ask how many finite subsets \(A\subseteq S\) have the left quotient set larger than the right quotient set. If $S$ is symmetric (i.e., $w\in S$ if and only if $w^{-1}\in S$), the number of subsets with the left quotient set larger than the right equals the number of subsets with the right quotient set larger than the left. Indeed, if
\(|A\inv A| > |AA\inv|\), then we can replace \(A\) with \(B\coloneqq A\inv\)
to get \(|B\inv B| < |BB\inv|\), and vice-versa. In particular, this implies
that the quantity
\[
    \mathbb{E}[|A\inv A|-|AA\inv|] \ ~=~ \  0,
\]
where $A\subseteq S$ is chosen by including elements at random with probability $1/2$. On the other hand, the variance
\[
    \mathrm{Var}(|A\inv A|-|AA\inv|) \ ~=~ \  \mathbb{E}\left[(|A\inv A|-|AA\inv|)^2\right]
\]
is nonzero for groups where \(|AA\inv| \neq |A\inv A|\) for some finite
subset \(A\subseteq S\). This suggests the following question in the free group $F_2$.

\begin{question}
    Let \(B_N \subseteq F_2\) be the set of words of length no greater than \(N\).
    With respect to the uniform probability measure on the subsets of
    \(B_N\), what is \(\mathrm{Var}(|A\inv A| - |AA\inv|)\)?
\end{question}

When studying the moments of $|AA^{-1}|$, one must handle dependence between the event $x\in AA^{-1}$ and the event $y\in AA^{-1}$. A technique for handling this dependence is via constructing the graph $C$ with vertex set $V(C)=S$ and edge set $E(C)=\{(u,v)\in S\times S\mid uv^{-1}=w_m\text{ or } vu^{-1}=w_m\text{ for some }m=1,\dots,k\}$. Then, $\mathbb{P}(w_1,\dots,w_k\notin AA^{-1})$ is the probability that $A$ is chosen to be a vertex cover of $C$. This technique is described in \cite{MissingSumsInSumset} for finite sets of integers. 



There are other unanswered questions about the values \(|AA^{-1}|-|A^{-1}A|\) can take in various groups.
\begin{question}[Answered in \cite{mouli2025classificationfinitegroupsequal} and \cite{Herzog-Kaplan-Longobardi-Maj}]\label{question: necessary and sufficient for A in G MLTR}
    What are the necessary and sufficient conditions on a group \(G\)
    so that there exists a finite subset \(A \subseteq G\) such that
    \(|A\inv A|\neq|AA\inv|\)?
\end{question}

Of course, $G$ being non-abelian is a necessary condition for $|A^{-1}A|\neq|AA^{-1}|$. However, one can check that $|AA^{-1}|=|A^{-1}A|$ for all subsets $A$ of the symmetric group $S_3$. Hence, being non-abelian is not a sufficient condition. 


\begin{question}[Answered in \cite{mouli2025classificationfinitegroupsequal} and \cite{Herzog-Kaplan-Longobardi-Maj}]\label{question: infinite family non-isomorphic}
    Is there an infinite family of finite, non-isomorphic,
    non-abelian groups \(\{G_i\}\) for which
    \(|A\inv A| \ = \ |AA\inv|\) for all subsets $A$ of $G_i$?
\end{question}

Following the initial draft of this paper, the sixth author answered Questions \ref{question: necessary and sufficient for A in G MLTR} and \ref{question: infinite family non-isomorphic} in the preprint \cite{mouli2025classificationfinitegroupsequal}. Shortly following this preprint, the authors of the present paper were directed to \cite{Herzog-Kaplan-Longobardi-Maj} which answered the questions previously and independently. Both answer the question by providing a classification of groups $G$ for which $|AA^{-1}|=|A^{-1}A|$ for all $A\subseteq G$. Question \ref{question: infinite family non-isomorphic} is answered affirmatively with the Hamiltonian $2$-groups $Q_{8}\times(C_2)^n$ being the unique infinite family of groups where $|AA^{-1}|=|A^{-1}A|$ holds for all subsets.

This classification appears as \cite[Theorem 7.4.]{Herzog-Kaplan-Longobardi-Maj} for all finite and infinite groups and as \cite[Theorem 1.1.]{mouli2025classificationfinitegroupsequal} where the classification is only for finite groups. We are grateful to Liubomir Chiriac for pointing us to \cite{Herzog-Kaplan-Longobardi-Maj} after reading a previous draft of this paper. Given that Questions \ref{question: necessary and sufficient for A in G MLTR} and \ref{question: infinite family non-isomorphic} are solved, we are interested in the following extension. Define the family of functions on groups $\{f_n\}_{n\geq1}$ given by $f_n(G)=\sup_{A\subseteq G,|A|\leq n}(|AA^{-1}|-|A^{-1}A|)$.
\begin{question}\label{question: conditions for f<c}
    Fix $c\geq 0$. What are the necessary and sufficient conditions on a group $G$ such that $f_n(G)\leq c$ for $n\geq 0$?
\end{question}

\begin{question}
    For what groups $G$ does $f_n(G)\to\infty$ as $n\to\infty$?
\end{question}

The $c=0$ case of Question \ref{question: conditions for f<c} is exactly Question \ref{question: necessary and sufficient for A in G MLTR} while the case $c>0$ asks for the maximal difference $|AA^{-1}|-|A^{-1}A|$ when $A\subseteq G$ is a subset of size $n$. We believe these questions could be a productive area for future research.

\newpage
\noindent
\textbf{Acknowledgements.}

\noindent The authors gratefully acknowledge the referee's careful reading of the manuscript and their constructive comments. The authors also thank Carl Pomerance and Mel Nathanson for stimulating discussions; the genesis of this work came from conversations the fifth-named author had with them at the Integers Conference in Georgia in 2025.
This project took place at the SMALL REU program at Williams College and was funded by the Finnerty Fund from Williams College and the National Science Foundation (Grant DMS2241623). The authors are grateful for the support of Amherst College, Princeton University, the University of Michigan, the University of Wisconsin, and Williams College.

\bibliographystyle{alpha}

\newcommand{\etalchar}[1]{$^{#1}$}

\end{document}